\documentclass[11pt]{article}
\usepackage{amsfonts,amsmath}
\usepackage{hyperref}
\usepackage{epsfig}
\usepackage{graphicx}

\newtheorem{theorem}{Theorem}[section]
\newtheorem{definition}[theorem]{Definition}
\newtheorem{proposition}[theorem]{Proposition}
\newtheorem{lemma}[theorem]{Lemma}
\newtheorem{claim}[theorem]{Claim}
\newtheorem{corollary}[theorem]{Corollary}
\newtheorem{remark}[theorem]{Remark}

\newtheorem{observation}[theorem]{Observation}

\newenvironment{proof}[1][Proof]{ \noindent \textbf{#1: }}{$\Box$
\bigskip}

\begin{document}

\title{Blockers for Non-Crossing Spanning Trees in Complete
Geometric Graphs}

\author{
Chaya Keller and Micha A. Perles \\
Einstein Institute of Mathematics, Hebrew University\\
Jerusalem 91904, Israel\\
$\mathrm{\{ckeller,perles\}}$@math.huji.ac.il\\
\\
Eduardo Rivera-Campo and Virginia Urrutia-Galicia \\
Departamento de Matem$\mathrm{\acute{a}}$ticas, Universidad
Aut$\mathrm{\acute{o}}$noma Metropolitana-Iztapalapa \\
Av. San Rafael Atlixco 186,
M$\mathrm{\acute{e}}$xico D.F. 09340, Mexico\\
$\mathrm{\{erc,vug\}}$@xanum.uam.mx
}

\maketitle

\begin{abstract}
In this paper we present a complete characterization of the
smallest sets that block all the simple spanning trees (SSTs) in a
complete geometric graph. We also show that if a subgraph is a
blocker for all SSTs of diameter at most $4$, then it must block
all simple spanning subgraphs, and in particular, all SSTs. For
convex geometric graphs, we obtain an even stronger result: being
a blocker for all SSTs of diameter at most $3$ is already
sufficient for blocking all simple spanning subgraphs.
\end{abstract}

\section{Introduction}

A {\it geometric graph} is a graph whose vertices are points in
general position in the plane,\footnote{Formally, the assumption
is that an edge never contains a vertex in its relative interior.
In the case of a complete geometric graph which we consider in
this paper, this implies that the vertices are in general position
(i.e., that no three vertices lie on the same line).} and whose
edges are segments connecting pairs of vertices. Let
$G=(V(G),E(G))$ be a complete geometric graph and let
$\mathcal{F}$ be a family of subgraphs of $G$. We say that a
subgraph $B$ of $G$ \textit{blocks } $\mathcal{F}$ if it has at
least one edge in common with each member of $\mathcal{F}$. We
denote by $\mathcal{B \left( F\right)} $ the collection of all
smallest (i.e., having the smallest possible number of edges)
subgraphs of $G$ that block $\mathcal{F}$, and call its elements
\textit{blockers} of $\mathcal{F}$.

Blockers for several families of subgraphs were studied in
previous papers. For example, the set $\mathcal{B}(SPM)$ of
blockers for the family of all simple (i.e., non-crossing) perfect
matchings in a complete convex geometric graph of even order was
characterized in~\cite{Maamar-Master}, and the family of
corresponding co-blockers (i.e., $\mathcal{B(B} (SPM))$) was
characterized in~\cite{Co-blockers}. The characterizations give
raise to interesting structures, such as classes of
\emph{caterpillars}~\cite{Caterpillar0,Caterpillar1}.

In this paper we study the set $\mathcal{B}\left( SST\right) $ of
blockers for the family of simple spanning trees (SSTs) of a
complete geometric graph, and give the following characterization:
\begin{definition}
A simple spanning subgraph $B$ of a complete geometric graph $G$
is a \emph{comb} of $G$ if:
\begin{enumerate}
\item The intersection of $B$ with the boundary of
$\mathrm{conv}(G)$ is a simple path $P$. We call this path the
\emph{spine} of $B$.

\item Each vertex in $V(G) \setminus P$ is connected by a unique
edge to an interior vertex of $P$.

\item For each edge $e$ of $B$, the line $l(e)$ spanned by $e$
does not cross any edge of $B$.\footnote{Note that the line $l(e)$
avoids all vertices of $G$ except the endpoints of $e$, since
$V(G)$ is in general position.}
\end{enumerate}
\end{definition}

\begin{figure}[tb]
\begin{center}
\scalebox{0.8}{
\includegraphics{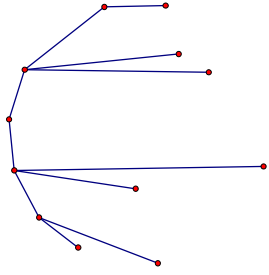}
} \caption{A comb in a non-convex geometric graph on $12$
vertices.} \label{Fig:Comb}
\end{center}
\end{figure}

Note that a comb $B$, regarded as an abstract tree, is a
caterpillar, and that the derived tree is the path $P$ with the
first and last edges removed. An example of a comb is shown in
Figure~\ref{Fig:Comb}.

\begin{theorem}\label{Thm:Main}
A graph $B$ is a blocker for the family of all simple spanning
trees of a complete geometric graph $G$ if and only if $B$ is
either a star (i.e., the set of all edges in $G$ that emanate from
a single vertex) or a comb of $G$.
\end{theorem}

We note that in the convex case, this characterization can be
derived by combining a result of Hernando~\cite{Hernando} that
characterizes those SSTs that meet all other SSTs, with a result
of K$\mathrm{\acute{a}}$rolyi et al.~\cite{Pach} that shows that
any two-coloring of a complete geometric graph contains a
monochromatic SST. Theorem~\ref{Thm:Main} was recently used
in~\cite{Rivera} to show that if $G$ is a complete geometric graph
with $n$ vertices in which exactly one vertex does not lie on the
boundary of $\mathrm{conv}(G)$, and $c$ is a coloring of the edges
of $G$ with $n(n-1)/2 - n + 1$ colors, then $G$ has a simple
spanning tree all of whose edges have different colors.

\medskip

\noindent We also present several refinements of
Theorem~\ref{Thm:Main}.

\medskip For a complete geometric graph $G$ and for $k \in \mathbb{N}$,
denote by $\mathcal{T}_{\leq k}\left( G\right) $ the family of all
simple spanning trees of $G$ with diameter at most $k$.
\begin{theorem}\label{Thm:General}
Let $B$ be a subgraph of a complete geometric graph $G$. If $B\in \mathcal{B}
\left( \mathcal{T}_{\leq 4}\left( G\right) \right) $,
then $B$ is either a star or a comb of $G$.
\end{theorem}

In the case of complete convex geometric graphs we can replace
diameter 4 by diameter 3, as follows:

\begin{theorem}\label{Thm:Convex}
Let $B$ be a subgraph of a complete convex geometric graph $G$. If $B\in
\mathcal{B}\left( \mathcal{T}_{\leq 3}\left( G\right) \right) $, then $B$ is a
comb of $G$.
\end{theorem}

The two latter results improve Theorem~\ref{Thm:Main} by showing
that being a blocker for SSTs of diameter at most 4 (or even at
most 3 in the convex case) is sufficient for being a blocker for
all SSTs. These results are tight in the sense that
$\mathcal{T}_{\leq 4}\left( G\right) $ cannot be replaced by
$\mathcal{T}_{\leq 3}\left( G\right) $ in
Theorem~\ref{Thm:General}, as we show by an example in
Section~\ref{sec:general}, and $\mathcal{T}_{\leq 3}\left(
G\right) $ cannot be replaced by $\mathcal{T}_{\leq 2}\left(
G\right) $ in Theorem~\ref{Thm:Convex}, since any spanning
subgraph blocks all trees in $\mathcal{T}_{\leq 2}\left( G\right)
$ but not all SSTs of $G $.

\medskip

Finally, the following result improves Theorem~\ref{Thm:Main} in
the opposite direction. We say that $H \subset G$ is a
\emph{simple spanning subgraph} (SSS) of $G$ if $H$ is
non-crossing and has no isolated vertices, i.e., every vertex of
$G$ is incident to an edge of $H$.
\begin{theorem}\label{Thm:SSS}
Let $B$ be a subgraph of a complete geometric graph $G$. If $B$ is
a star or a comb of $G$, then $B$ blocks all simple spanning
subgraphs of $G$.
\end{theorem}

The paper is organized as follows: In
Section~\ref{sec:definitions} we give precise definitions and
notations used throughout the paper. In Section~\ref{sec:diam-3}
we prove properties of blockers for $\mathcal{T}_{\leq 3}\left(
G\right) $ common to the general case and the convex case. In
Sections~\ref{sec:convex} and~\ref{sec:general} we prove
Theorems~\ref{Thm:Convex} and~\ref{Thm:General}, respectively, and
in Section~\ref{sec:other-direction} we complete the proof of
Theorem~\ref{Thm:Main} and prove Theorem~\ref{Thm:SSS} by showing
that $\mathcal{B \subset B}\left( SSS\right) $ for any complete
geometric graph $G$, where $\mathcal{B}$ denotes the family of all
combs of $G$.

\section{Definitions and Notations}
\label{sec:definitions}

In this section we present some definitions and notations we use
in the paper.

\medskip

\noindent \textbf{Geometric graphs.} Throughout the paper, $G$ is
a complete geometric graph on $n$ vertices. The sets of vertices
and edges of $G$ are denoted by $V(G)$ and $E(G)$, respectively.
The convex hull of $V(G)$ is denoted by $\mathrm{conv}(G)$.
Vertices in $V(G)$ and edges in $E(G)$ that lie on the boundary of
$\mathrm{conv}(G)$ are called {\it boundary vertices} and {\it
boundary edges} of $G$, respectively. A geometric graph is {\it
simple} if it does not contain a pair of crossing edges. For more
information on geometric graphs, the reader is referred
to~\cite{Pach2}.

\medskip

\noindent \textbf{Caterpillars.} Throughout the paper, $T$ is a
tree. A tree $T$ is a {\it caterpillar} if the derived graph $T'$
(i.e., the graph obtained from $T$ by removing all leaves and
their incident leaf edges) is a path (or is empty). A longest path
in a caterpillar $T$ is called a {\it spine} of $T$. (Note that
any edge of $T$ either belongs to every spine or is a leaf edge of
$T$.) If the diameter of $T$ is 3, then $T$ contains an edge
$[x,y]$ such that each vertex in $V(T)$ is at distance at most 1
from either $x$ or $y$. Such an edge $[x,y]$ is called the {\it
central edge} of $T$. (Note that any tree of diameter 3 is a
caterpillar.)

\medskip

\noindent \textbf{General notations in the plane.} Throughout the
paper, $l$ denotes a line. Each line $l$ partitions the plane into
open half-planes. We denote them by $l^{+}$ and $l^{-}$, and call
them the {\it sides} of the line. The unique line that contains
two points $a,b \in \mathbb{R}^2$ is denoted by $l(a,b)$. The
complement of a set $A$ in $\mathbb{R}^2$, i.e., the set
$\mathbb{R}^2 \setminus A$, is denoted by $A^c$.


\section{Some Properties of Blockers for SSTs of Diameter at Most $3$}
\label{sec:diam-3}

In this section we establish several properties of blockers for
SSTs of diameter at most $ 3$. First we show that the number of
edges in these blockers is $n-1$ (where $n$ is the number of
vertices in $G$), and then we show that all such blockers are
caterpillars.

\subsection{The Size of the Blockers}
\label{sec:sub:size}

\begin{proposition}\label{Prop:size}
Let $G$ be a complete geometric graph on $n$ vertices. Then the
size (i.e., number of edges) of the blockers for SSTs of diameter
at most $ 3$ in $G$ is $n-1$.
\end{proposition}

As any star blocks all SSTs of diameter at most $ 3$ (and actually
even all spanning subgraphs), the size of the blockers is at most
$n-1$. The other inequality is a consequence of the following
unpublished result of Perles~\cite{Micha}.
\begin{theorem}\label{Thm:Micha}
Let $G_1$ be a geometric graph on $n$ vertices. If
$|E(\overline{G_1})| \leq n-2$, then $G_1$ includes an SST of
diameter at most $ 3$. ($\overline{G_1}$ denotes the graph
complementary to $G_1$, on the same set of vertices.)
\end{theorem}

Theorem~\ref{Thm:Micha} implies that a set of at most $ n-2$ edges
cannot block all the SSTs of diameter at most $ 3$, since its
complement includes such an SST. Thus, the size of blockers is at
least $n-1$, which completes the proof of
Proposition~\ref{Prop:size}. For the sake of completeness we
present here the proof of Theorem~\ref{Thm:Micha}.

\medskip

\begin{proof}
Since $|E(\overline{G_1})| \leq n-2$, the following two statements
hold:
\begin{enumerate}
\item $\overline{G_1}$ has a vertex of degree 0 or 1.

\item $\overline{G_1}$ is not connected.
\end{enumerate}

Since the number of edges in $\overline{G_1}$ is smaller than the
number of vertices, at least one connected component of
$\overline{G_1}$ is a tree. Denote one such component by $A$. If
$A$ is a single vertex $x'$, then $G_1$ contains the star centered
at $x'$, which is an SST of diameter 2 (assuming $n \geq 2$).
Otherwise, $A$ has a leaf $x$. Denote the (only) neighbor of $x$
in $A$ by $y$.

Consider the ray $\vec{xy}$, and turn it around $x$ until it hits
a vertex $b \not \in A$ for the first time. (There must be a
vertex $b \not \in A$, since $\overline{G_1}$ is not connected,
and thus $A \neq V(G)$. Moreover, $b$ is unique, since a ray
emanating from $x$ cannot contain two other vertices of $G$.)
Denote by $H$ the closed convex cone bounded by the rays
$\vec{xy}$ and $\vec{xb}$.

\begin{figure}[tb]
\begin{center}
\scalebox{0.8}{
\includegraphics{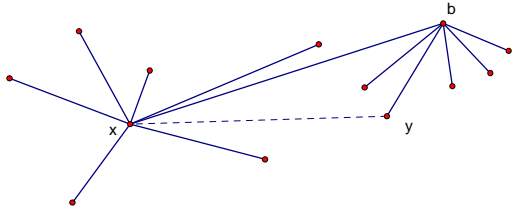}
} \caption{An SST of diameter 3 that avoids $\overline{G_1}$.}
\label{Fig:Micha}
\end{center}
\end{figure}

Let $T$ be the subgraph of $G$ whose edges are
\[
\{[x,z]:z \in V(G) \cap H^{c}\} \cup \{[b,w]:w \in (V(G) \setminus
\{b\}) \cap H\},
\]
as illustrated in Figure~\ref{Fig:Micha}. It is clear by the
construction that $T$ is an SST of diameter 3 (with central edge
$[x,b]$), or of diameter 2 (if $V(G) \subset H$). We claim that $T
\subset G_1$. Indeed, the edges $[x,z]$, where $z \in H^c \cap
V(G)$ are all in $G_1$, as the only edge in $\overline{G_1}$ that
contains $x$ is $[x,y]$. The edges $[b,w]$, where $w \in H \cap
(V(G) \setminus \{b\})$ are also in $G_1$, since the vertices $\{w
: w \in (V(G) \setminus \{b\}) \cap H\}$ belong to $A$, whereas
$b$ belongs to another connected component of $\overline{G_1}$.
Therefore, $T \subset G_1$, which completes the proof.
\end{proof}

We note that the proof of Theorem~\ref{Thm:Micha} implies a
stronger statement:
\begin{proposition}\label{Prop:Tree}
Let $G$ be a complete geometric graph on $n$ vertices. Then the
blockers for SSTs of diameter at most $ 3$ in $G$ are spanning
trees.
\end{proposition}

\begin{proof}
Let $B$ be a blocker for SSTs of diameter at most $ 3$ in $G$. By
Proposition~\ref{Prop:size}, $|E(B)|=n-1$. It is clear that $B$ is
a spanning subgraph of $G$ without isolated vertices, since
otherwise it avoids a star which is an SST of diameter 2. If $B$
is not a tree, then the two statements at the beginning of the
proof of Theorem~\ref{Thm:Micha} clearly hold (i.e., $B$ has a
vertex of degree 0 or 1, and is not connected). Thus, by the proof
of Theorem~\ref{Thm:Micha}, $\bar{B}$ contains an SST of diameter
at most $ 3$, a contradiction.
\end{proof}

\subsection{The Blockers are Caterpillars}
\label{sec:sub:caterpillar}

In order to further characterize the blockers, we use two
observations.
\begin{observation}\label{Obs:Basic}
Let $B$ be a subgraph of $G$. Assume there exist two vertices $a,b
\in V(G)$ and a line $l$, such that:
\begin{enumerate}
\item $[a,b] \not \in E(B)$, and

\item $a$ and all neighbors of $b$ in $B$ lie on one (open) side
$l^{+}$ of $l$, while $b$ and all neighbors of $a$ in $B$ lie on
the other (open) side $l^{-}$ of $l$.
\end{enumerate}
Then $B \not \in \mathcal{B(T}_{\leq 3})$.
\end{observation}

\begin{figure}[tb]
\begin{center}
\scalebox{0.8}{
\includegraphics{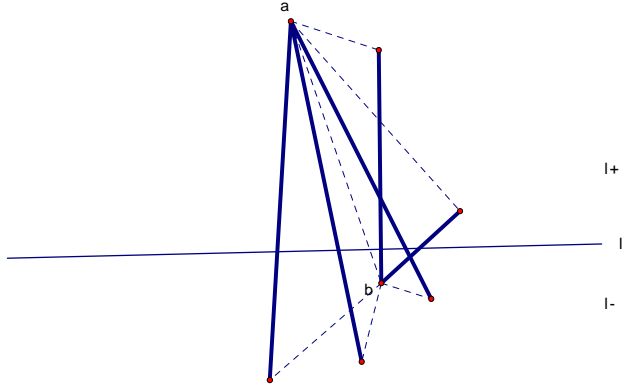}
} \caption{An illustration of the proof of
Observation~\ref{Obs:Basic}. The edges of $B$ are depicted by
heavy lines, and the edges of $T$ by dotted lines.}
\label{Fig:Observation1}
\end{center}
\end{figure}

\begin{proof}
If Conditions~(1) and~(2) hold, then $B$ avoids the following SST,
as illustrated in Figure~\ref{Fig:Observation1}: $T=(V(G),E(T))$,
where
\[
E(T)=\{[a,b]\} \cup \{[a,x]: x \in V(G) \cap (l^{+} \cup l), x
\neq a\} \cup \{[b,y]: y \in V(G) \cap l^{-}, y \neq b \}.
\]
It is clear that $\mathrm{diam}(T) \leq 3$ (as the distance (in
$T$) of all vertices from the edge $[a,b]$ is at most 1), that $T$
is crossing-free, and that $T$ avoids $B$. The assertion follows.
\end{proof}

\begin{remark}
It is clear that the observation holds also if $a$ and $b$ have
neighbors on $l$, as long as they do not have a common neighbor on
$l$.
\end{remark}

\begin{corollary}\label{Cor1}
Assume $B \in \mathcal{B(T}_{\leq 3})$, and let $a,b$ be two
leaves of $B$. Let the corresponding leaf edges be $[a,c]$ and
$[b,d]$. If $a,b,c,d$ are mutually distinct, then the points
$a,b,d,c$ (in this order) are the vertices of a convex
quadrilateral.
\end{corollary}


\begin{proof}
Recall that $a,b,c,d$ are vertices of $G$ and therefore, are in
general position. If they are not in convex position, then the
segments $[a,d]$ and $[b,c]$ are disjoint. The same holds if they
form a convex quadrilateral in a different order, say $a,c,b,d$ or
$a,b,c,d$. If $[a,d] \cap [b,c] = \emptyset$, then these two
segments can be separated by a line $l$. This means that the
conditions of Observation~\ref{Obs:Basic} hold for $a,b,$ and $l$,
and thus $B \not \in \mathcal{B(T}_{\leq 3})$, a contradiction.
\end{proof}

\begin{observation}\label{Obs2}
Suppose $B \in \mathcal{B(T}_{\leq 3})$. Let $a$ be a leaf of $B$
with leaf edge $[a,b]$. Let $c \in V(G) \cap l(a,b)^{+}$ be the
vertex for which the angle $\angle abc$ is maximal (among all
vertices in $l(a,b)^{+}$). Then $[b,c] \in B$.
\end{observation}

\begin{figure}[tb]
\begin{center}
\scalebox{0.8}{
\includegraphics{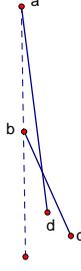}
} \caption{An illustration to the proof of
Observation~\ref{Obs2}.} \label{Fig:Observation2}
\end{center}
\end{figure}

\begin{proof}
If $[b,c] \not \in B$, then $B$ avoids the following SST:
\[
T= \{[a,x]: x \in V(G) \setminus \{a,b\} \} \cup \{[b,c]\}.
\]
It is clear that $T$ is a spanning tree of diameter at most $ 3$.
$T$ avoids $B$ since the only edge in $B$ that emanates from $a$
is $[a,b]$, and since $[b,c] \not \in B$. Finally, $T$ is simple,
since if two edges in $T$ cross, then these must be edges of the
form $[a,d]$ and $[b,c]$ for some $d \in V(G)$, and in such case,
$\angle abd > \angle abc$ (see Figure~\ref{Fig:Observation2}),
contradicting the choice of $c$. Thus, $B \not \in
\mathcal{B(T}_{\leq 3})$, a contradiction.
\end{proof}

Clearly, the same holds for the vertex $c \in V(G) \cap
l(a,b)^{-}$ for which the angle $\angle abc$ is maximal among all
vertices in $l(a,b)^{-}$.

If the leaf edge $[a,b]$ lies on the boundary of
$\mathrm{conv}(G)$, then the line $l(a,b)$ supports $V(G)$, and
thus only one of the sides of $l(a,b)$ (w.l.o.g. $l(a,b)^{+}$)
contains vertices of $G$. The vertex $c \in V(G) \cap l(a,b)^{+}$
for which the angle $\angle abc$ is maximal is the vertex that
follows $b$ on the boundary of $\mathrm{conv}(G)$, and thus
$[b,c]$ is a boundary edge.

If $[a,b]$ is not a boundary edge, then there exist two vertices
$c \in V(G) \cap l(a,b)^{+}$ and $c' \in V(G) \cap l(a,b)^{-}$
such that the angles $\angle abc, \angle abc'$ are maximal (each
with respect to its side of $l(a,b)$), and $[b,c],[b,c'] \in B$.
This observation is used in the proof of the theorem below.

\medskip

Now we are ready to present the main result of this section.
\begin{theorem}\label{Thm:Caterpillar}
Let $G$ be a complete geometric graph. Then any blocker for SSTs
of diameter at most $ 3$ in $G$ is either a star or a caterpillar
with a spine whose terminal edges lie on the boundary of
$\mathrm{conv}(G)$.\footnote{Note that usually the term ``terminal
edges of the spine'' of a caterpillar is not defined uniquely.
Here and in the sequel we mean that there exists a spine whose
terminal edges are boundary edges, and in all proofs where we
consider {\it the} spine of $B$, we refer to a particular spine
whose terminal edges are boundary edges.}
\end{theorem}

\begin{figure}[tb]
\begin{center}
\scalebox{0.8}{
\includegraphics{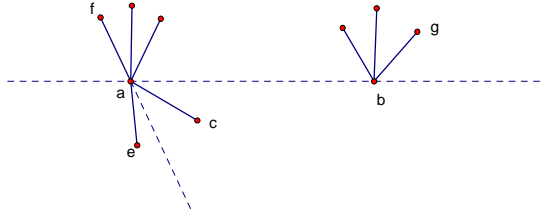}
} \caption{An illustration to the proof of
Theorem~\ref{Thm:Caterpillar}.} \label{Fig:Caterpillar}
\end{center}
\end{figure}

\begin{proof}
Suppose $B \in \mathcal{B(T}_{\leq 3})$. By
Proposition~\ref{Prop:Tree}, $B$ is a tree. If $B$ is a star, we
are done. Otherwise, the derived graph $B'$ is a tree with more
that one vertex, and thus it has at least two leaves. Let $a,b$ be
distinct leaves of $B'$. By Corollary~\ref{Cor1} (with $a,b$
playing the role of $c,d$), all the leaf edges of $B$ that emanate
from $a$ and $b$ lie on the same side of the line $l(a,b)$,
w.l.o.g., $l(a,b)^{+}$ (see Figure~\ref{Fig:Caterpillar}). We
claim that the extremal leaf edges emanating from $a$ and $b$,
denoted in the figure by $[a,f]$ and $[b,g]$, are boundary edges.

Assume on the contrary that $[a,f]$ is not a boundary edge. As
described above, it follows from Observation~\ref{Obs2} that if $c
\in l(a,f)^{+}$ with $\angle fac$ maximal, and $e \in l(a,f)^{-}$
with $\angle fae$ maximal, then $[a,c],[a,e] \in B$. (Here
$l(a,f)^{+}$ is the side of $l(a,f)$ that contains $b$.) We have
$c \not \in l(a,b)^{+}$, as otherwise $\angle fac < \angle fab$,
contradicting the choice of $c$. Thus, $[a,c]$ is not a leaf edge
(since all the leaf edges that emanate from $a$ lie in
$l(a,b)^{+}$). On the other hand, since $a$ is a leaf of $B'$, all
the edges in $B$ that emanate from $a$ except one are leaf edges,
and thus, $[a,e]$ is a leaf edge. Therefore, $e \in l(a,b)^{+}$
(and not as shown in the figure), which contradicts the assumption
that $[a,f]$ is the extremal (i.e., the leftmost) leaf edge
emanating from $a$.

So far we have shown that any leaf of $B'$ is contained in a leaf
edge of $B$ that is a boundary edge. In order to complete the
proof, it suffices to show that $B'$ has only two leaves, which
will imply that $B'$ is a path, and hence $B$ is a caterpillar. As
the terminal edges of the spine of a caterpillar $B$ emanate from
the two leaves of $B'$, it will follow that these edges can be
chosen to be boundary edges of $G$, completing the proof of the
theorem.

Assume on the contrary that $B'$ has at least three leaves, say
$a,b,$ and $c$. By the previous steps of the proof, $B$ has leaf
edges $[a,d],[b,e],$ and $[c,f]$, which all lie on the boundary of
$\mathrm{conv}(G)$. By Corollary~\ref{Cor1}, these edges must
satisfy the following three conditions:
\begin{itemize}
\item $d$ and $e$ lie on the same side of $l(a,b)$.

\item $d$ and $f$ lie on the same side of $l(a,c)$.

\item $e$ and $f$ lie on the same side of $l(b,c)$.
\end{itemize}

\begin{figure}[tb]
\begin{center}
\scalebox{0.8}{
\includegraphics{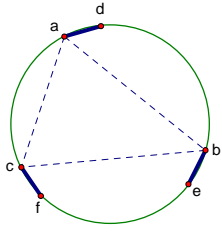}
} \caption{Three leaf edges on the boundary of
$\mathrm{conv}(G)$.} \label{Fig:3-Leaves}
\end{center}
\end{figure}

But if we somehow orient the boundary of $\mathrm{conv}(G)$ (say,
counterclockwise), then at least two of the directed edges
$\overrightarrow{ad}, \overrightarrow{be}$ and
$\overrightarrow{cf}$ point the same way (both forward or both
backward), and thus fail the condition above. This completes the
proof of the theorem.
\end{proof}

\section{The Convex Case -- Proof of Theorem~\ref{Thm:Convex}}
\label{sec:convex}

\begin{figure}[tb]
\begin{center}
\scalebox{0.8}{
\includegraphics{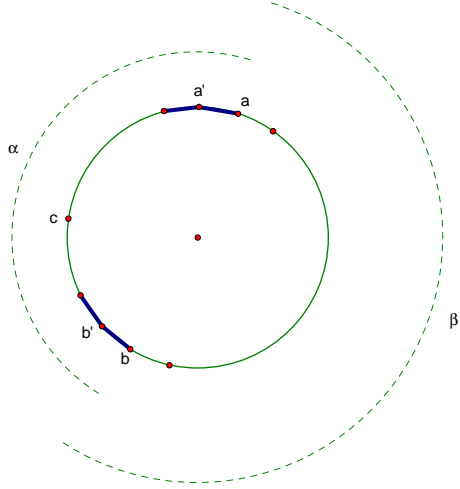}
} \caption{The terminal edges of the spine of a blocker for SSTs
of diameter at most $ 3$ in a complete convex geometric graph
divide the boundary into two arcs.} \label{Fig:Conv1}
\end{center}
\end{figure}

In this section we assume in addition that $G$ is convex, i.e.,
that the vertices of $G$ are in convex position in the plane. By
Theorem~\ref{Thm:Caterpillar}, a blocker for SSTs of diameter at
most $ 3$ in $G$ is a caterpillar, and the terminal edges of its
spine lie on the boundary of $\mathrm{conv}(G)$. We wish to show
that the whole spine of $B$ lies on the boundary of
$\mathrm{conv}(G)$. This is clear when $B$ is a star. Assume,
therefore, that $B$ is not a star. Denote the terminal edges by
$[a,a']$ and $[b',b]$, where $a$ and $b$ are the leaves, and $a'
\neq b'$, as shown in Figure~\ref{Fig:Conv1}. (Note that by
Corollary~\ref{Cor1}, $a$ and $b$ lie on the same side of the line
$l(a',b')$.) Let $\alpha$ and $\beta$ denote the two closed arcs
with endpoints $a,b$ on the boundary of $\mathrm{conv}(G)$, as
shown in the figure.
\begin{claim}\label{Claim:Conv1}
If $c$ is a leaf of $B$, then $c \in \beta$.
\end{claim}

\begin{proof}
Assume, w.l.o.g., that $c \neq a$ and $c \neq b$. Denote the leaf
edge of $B$ that emanates from $c$ by $[c,c']$. By
Corollary~\ref{Cor1}, the following two conditions hold:
\begin{itemize}
\item Either $a'=c'$ or $a$ and $c$ lie on the same side of $l(a',c')$.

\item Either $b'=c'$ or $b$ and $c$ lie on the same side of $l(b',c')$.
\end{itemize}

If $c \in \alpha$ (as in the figure), then the two conditions
clearly contradict each other. Hence, $c \in \beta$, as claimed.
\end{proof}

Now we are ready to prove the main result of this section.
\begin{theorem}\label{Thm:Main-convex1}
Any blocker for SSTs of diameter at most $ 3$ in a complete convex
geometric graph $G$ is a comb of $G$.
\end{theorem}

\begin{figure}[tb]
\begin{center}
\scalebox{0.8}{
\includegraphics{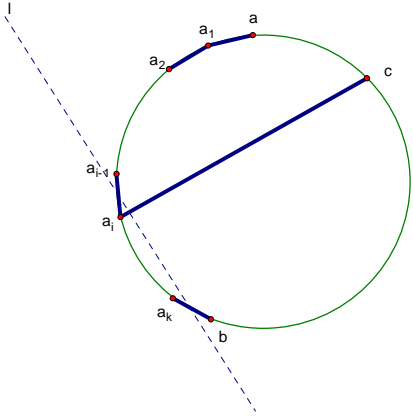}
} \caption{An illustration of the proof of
Theorem~\ref{Thm:Main-convex1}.} \label{Fig:Conv2}
\end{center}
\end{figure}

\begin{proof}
Following the notations of Figure~\ref{Fig:Conv1}, denote the
vertices of $G$ on the arc $\alpha$ by $a=a_0,
a'=a_1,a_2,a_3\ldots, a_k=b',a_{k+1}=b$, as shown in
Figure~\ref{Fig:Conv2}. Since any vertex of a caterpillar that
does not belong to its spine is a leaf, it follows from
Claim~\ref{Claim:Conv1} that all the vertices
$a_0,a_1,\ldots,a_{k+1}$ belong to the spine of $B$. Let the spine
of $B$ be $\langle d_0, d_1, \ldots, d_l,d_{l+1} \rangle$, where
$d_0=a_0=a$, $d_1=a_1=a'$, $d_l=a_k=b'$, and $d_{l+1}=a_{k+1}=b$.
If the spine of $B$ is not $\langle a_0, a_1, \ldots, a_k,a_{k+1}
\rangle$, then there is a first index $\nu$, $1 \leq \nu \leq
l-1$, such that either $d_{\nu +1} \not \in \alpha$, or $d_{\nu
+1} \in \alpha$, but $d_{\nu+1}$ precedes $d_{\nu}$ on $\alpha$.
Assume $d_{\nu}=a_i$, $0<i<k$. Let $l$ be a line that crosses the
segments $[a_{i-1},a_i]$ and $[a_k,b]$ (see
Figure~\ref{Fig:Conv2}). On one side of $l$ (call it $l^{-}$), we
have $a_i$ and the only neighbor of $b$ in $B$, i.e., $a_k$. On
the other side $l^{+}$ of $l$ we have $b$, all the vertices
$a_0,a_1,\ldots,a_{i-1}$, and the whole path $\beta$. The only
neighbors of $a_i (=d_{\nu})$ in $B$ are its predecessor $d_{\nu
-1}$ (which lies in $\{a_0,a_1,\ldots,a_{i-1}\}$), its successor
$d_{\nu +1}$ (which lies either in $\{a_0,a_1,\ldots,a_{i-1}\}$ or
in $\beta$), and possibly some leaves (which all lie in $\beta$).
Thus, all neighbors of $a_i$ in $B$ lie in $l^{+}$. Hence, the
conditions of Observation~\ref{Obs:Basic} hold for $d_{\nu}=a_i,b$
and the line $l$, and thus, by the observation, $B$ avoids an SST
of diameter at most $ 3$, a contradiction.

Therefore, the spine of $B$ is the boundary path $\langle a_0,
a_1, \ldots, a_k, b \rangle$. Finally, $B$ is simple, since the
only edges in $B$ that can cross are the leaf edges, and these
edges do not cross, again by Corollary~\ref{Cor1}. This completes
the proof.
\end{proof}

\section{The General Case -- Proof of Theorem~\ref{Thm:General}}
\label{sec:general}

In this section we consider again general geometric graphs. It
turns out that in this case, blocking SSTs of diameter at most $
3$ is not sufficient even for blocking SSTs of diameter $4$, as we
demonstrate by an example at the end of this section. Thus, we
strengthen the assumption, and assume now that $B$ is a blocker
for all SSTs of diameter at most $ 4$. This allows us to use the
following observation.
\begin{lemma}\label{Obs3}
Let $B$ be a blocker for SSTs of diameter at most $ 4$ in $G$. Let
$b$ be a boundary vertex of $G$, and let $[a,b]$ and $[b,c]$ be
the two boundary edges of $G$ that contain $b$. If at least one of
these edges is not in $B$, then $b$ is a leaf of $B$.
\end{lemma}

\begin{proof}
Clearly, the degree of $b$ in $B$ is at least 1, as otherwise, $B$
avoids the star centered at $b$. Assume, on the contrary, that the
degree of $b$ in $B$ is at least $2$, and that (w.l.o.g.) $[a,b]
\not \in B$. Let $G_1$ be the graph obtained from $G$ by omitting
the vertex $b$ and all edges that contain it, and let $B_1 = G_1
\cap B$. By the assumption, $B_1$ is a graph on $n-1$ vertices
(where $n=|V(G)|$) that has at most $n-3$ edges. Therefore, by
Theorem~\ref{Thm:Micha}, $B_1$ avoids an SST $T_1$ of diameter at
most $ 3$ in $G_1$. Since $[a,b]$ is a boundary edge of $G$, it
does not cross any edge of $T_1$, and thus $T=T_1 \cup \{[a,b]\}$
is an SST of $G$ of diameter at most $ 4$ that avoids $B$, a
contradiction.
\end{proof}

Now we are ready to prove the main result of this section.
\begin{theorem}
Let $G$ be a complete geometric graph, and let $B$ be a blocker
for SSTs of diameter at most $ 4$ in $G$. Then $B$ is either a
star or a comb of $G$.
\end{theorem}

\begin{proof}
Assume $B$ is not a star. By Theorem~\ref{Thm:Caterpillar}, $B$ is
a caterpillar, and it has a spine $\langle b_0,b_1, \ldots,
b_{k+1} \rangle$ $(k \geq 2)$ whose extreme edges $[b_0,b_1]$ and
$[b_{k},b_{k+1}]$ lie on the boundary of $\mathrm{conv}(G)$. We
would like to show that all the edges $[b_{i},b_{i+1}]$ are
boundary edges of $G$. Assume on the contrary that this is not
true, and let $i$, $0<i<k$, be the smallest index such that
$[b_{i},b_{i+1}]$ is not a boundary edge. By the assumption,
$[b_{i-1},b_i]$ is a boundary edge of $G$. Denote the other
boundary edge that contains $b_i$ by $[b_i,c]$. We claim that
$[b_i,c] \not \in B$. Indeed, as the spine edge $[b_i,b_{i+1}]$ is
not a boundary edge, if we had $[b_i,c] \in B$ then this edge
would be a leaf edge of $B$. But this is impossible, since by the
proof of Theorem~\ref{Thm:Caterpillar}, $B$ cannot contain three
{\it boundary} leaf edges. Therefore, $b_i$ satisfies the
condition of Lemma~\ref{Obs3}, and thus, by that lemma, $b_i$ is a
leaf of $B$, a contradiction.

So far we proved that the spine of $B$ lies on the boundary of
$\mathrm{conv}(G)$. Consequently, if two edges of $B$ cross, then
both must be leaf edges of $B$. However, a leaf edge of $G$ cannot
cross another leaf edge, by Corollary~\ref{Cor1}. Thus, $B$ is
simple. Finally, if an edge of $G$ crosses the line spanned by
another edge, then there are two possibilities:
\begin{enumerate}
\item Both edges are leaf edges. In this case, the convex hull of
the union of these two leaf edges is a triangle and not a
quadrilateral, contrary to Corollary~\ref{Cor1}.

\item One of the edges is a leaf edge $[b_i,d]$, and the line
$l(b_i,d)$ crosses the boundary edge $[b_j,b_{j+1}]$, for some
$i,j$. Assume, w.l.o.g., that $k>j>i$ (and thus, in particular,
$i<k-1$). Consider the edges $[b_i,d]$ and $[b_{k},b_{k+1}]$. Both
are leaf edges of $B$, and they lie on different sides of the line
$l(b_i,b_{k})$. This contradicts Corollary~\ref{Cor1}, since $B
\in \mathcal{B(T}_{\leq 4}) \subset \mathcal{B(T}_{\leq 3})$.
\end{enumerate}

This completes the proof of the theorem.
\end{proof}

\subsection{Blocking SSTs of Diameter at most $ 3$ is Insufficient}
\label{sec:sub:example}

\begin{figure}[tb]
\begin{center}
\scalebox{1.0}{
\includegraphics{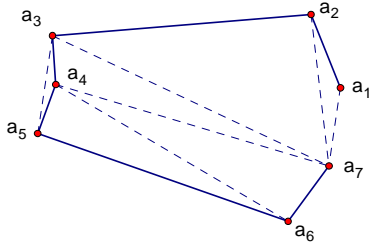}
} \caption{An example of a subgraph $B$ of a complete geometric
graph $G$ that blocks all SSTs of diameter at most $ 3$, but
avoids an SST $T$ of diameter 4. The full lines represent edges of
$B$, and the dotted lines represent edges of $T$.}
\label{Fig:Example}
\end{center}
\end{figure}

The example presented in Figure~\ref{Fig:Example} shows that
blocking SSTs of diameter at most $ 3$ is not sufficient even for
blocking SSTs of diameter 4. In the example, it is clear that $T$
(whose edges are represented by dotted lines) is an SST of
diameter 4, and that the path $B$ (whose edges are represented by
full lines) avoids it. It is also clear that $B$ blocks any SST of
diameter 2, since such SSTs are stars, and $B$ is a spanning
subgraph of $G$. In order to prove that $B$ meets all SSTs of
diameter 3, we show that no edge in $E(G)$ can be the {\it central
edge} of an SST of diameter 3 that avoids $B$. We do this using
the following observation.
\begin{lemma}\label{Obs4}
Suppose $[x,y] \in E(G) \setminus E(B)$. If there exist $z,w \in
V(G)$, such that:
\begin{enumerate}
\item The points $x,y,z,w$ are distinct and in convex position,

\item $[x,w],[y,z] \in E(B)$, and

\item The segments $[x,y]$ and $[z,w]$ do not cross,
\end{enumerate}
then $[x,y]$ cannot be the central edge in an SST of diameter 3
that avoids $B$.
\end{lemma}

\begin{proof}
Assume, on the contrary, that $[x,y]$ is the central edge of such
an SST $T$. Then $z$ and $w$ must be at distance 1 in $T$ from
$[x,y]$. As $[x,w],[y,z] \in E(B)$ and $T$ avoids $B$, this can
happen only if $[x,z],[y,w] \in E(T)$. However, since $x,y,z,w$
are in convex position and the pairs of edges $\{[x,y],[z,w]\}$
and $\{[x,w],[y,z]\}$ do not cross, the pair $\{[x,z],[y,w]\}$
must cross, contradicting the assumption that $T$ is simple.
\end{proof}

It can be seen, by checking all pairs $(i,j)$ with $1 \leq i<j
\leq 7$, that in the example, no edge $[a_i,a_j]$ can be the
central edge of an SST $T$ of diameter 3 that avoids $B$, since
for any edge $[a_i,a_j]$, at least one of the following holds:
\begin{enumerate}
\item $[a_i,a_j] \in E(B)$. This happens when $j=i+1$.

\item There exists a $k$ such that $[a_i,a_k] , [a_j,a_k] \in
E(B)$ (and thus, $a_k$ cannot be at distance 1 in $T$ from
$[a_i,a_j]$). This happens when $j=i+2$.

\item The vertices $x=a_i,y=a_j,z=a_{j+1},$ and $w=a_{i-1}$
satisfy the conditions of Lemma~\ref{Obs4}. This happens when
$1<i$ and $i+3 \leq j<7$. Note that in this case we never obtain
$\{3,4,5\} \subset \{i-1,i,j,j+1\}$.

\item The vertices $x=a_i,y=a_j,z=a_{j-1},$ and $w=a_{i+1}$
satisfy the conditions of Lemma~\ref{Obs4}. This happens when $1
\leq i < i+3 \leq j \leq 7$, and $i=1$ or $j=7$ (or both).
\end{enumerate}

\noindent Therefore, $B$ blocks all SSTs of diameter 3, as
asserted.

\medskip

\noindent We note that the example can be enlarged arbitrarily:
the edges $[a_1,a_2]$ and $[a_6,a_7]$ can be replaced by longer
convex polygonal arcs.

\section{The Converse Direction}
\label{sec:other-direction}

In this section we prove Theorem~\ref{Thm:SSS}, which is an
improved variant of the converse direction of Theorem~\ref{Thm:Main}.
\begin{theorem}\label{Thm:Converse}
Let $G$ be a complete geometric graph, and let $B \subset G$ be a
comb in $G$. Then $B$ meets every simple spanning subgraph of $G$.
\end{theorem}

\begin{proof}
Assume, on the contrary, that $H$ is a simple spanning subgraph of
$G$ such that $E(H) \cap E(B) = \emptyset$. Denote the spine of
$B$ by $\langle a_{-1}, a_0, a_1, \ldots, a_k, a_{k+1} \rangle$.
First we show that there is no loss of generality in assuming that
$H$ does not contain edges of the form $[a_i,a_j]$, where $-1 \leq
i,j \leq k+1$ (except, possibly, $[a_{-1},a_{k+1}]$).

Assume that $H$ contains such edges, and let $[a_{i_0},a_{j_0}]
\in H$, where $i_0<j_0$, be such an edge that minimizes the
difference $j-i$. Consider the subgraph of $G$ defined by:
\[
G_1=G \cap \mathrm{conv}(\{a_{i_0},a_{i_0+1},\ldots,a_{j_0}\}).
\]
Note that $H$ does not contain edges that connect vertices in
$V(G_1) \setminus \{a_{i_0},a_{j_0}\}$ with vertices in $V(G)
\setminus V(G_1)$, as such edges would cross the edge
$[a_{i_0},a_{j_0}]$, contradicting the simplicity of $H$. Thus,
the restriction of $H$ to
$\mathrm{conv}(\{a_{i_0},a_{i_0+1},\ldots,a_{j_0}\})$, i.e., $H_1
= H \cap G_1$, is a simple spanning subgraph of $G_1$. Similarly,
it is clear that $B_1 = B \cap G_1$ is a simple spanning
caterpillar of $G_1$ such that the line spanned by an edge $e$ of
$B_1$ never crosses another edge of $B_1$. Finally, $H_1$ does not
contain edges of the form $[a_i,a_j]$ for $i_0 \leq i,j \leq j_0$
(except for $[a_{i_0},a_{j_0}]$) by the minimality of the edge
$[a_{i_0},a_{j_0}]$. Thus, by reduction to $G_1,B_1$, and $H_1$,
we can assume w.l.o.g. that $H$ does not contain edges of the form
$[a_i,a_j]$ (except, possibly, for $[a_{-1},a_{k+1}]$).

\begin{figure}[tb]
\begin{center}
\scalebox{0.8}{
\includegraphics{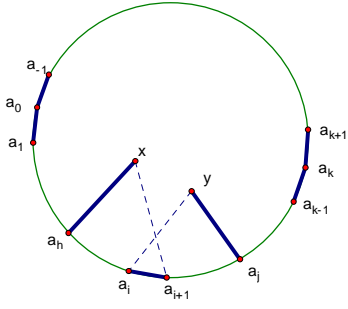}
} \caption{An illustration to the proof of
Theorem~\ref{Thm:Converse}. The edges of $B$ are represented by
full lines, and the edges of $H$ are represented by dotted lines.}
\label{Fig:Converse}
\end{center}
\end{figure}

Now, we define a function $f:\{0,1,\ldots,k\} \rightarrow
\{0,1,\ldots,k\}$ by the following procedure, performed for each
$0 \leq i \leq k$.
\begin{enumerate}
\item Consider the vertex $a_i$. Since $H$ is a spanning subgraph,
there exist edges in $E(H)$ that emanate from $a_i$. Pick one such
edge $[a_i,y]$. Note that $y \neq a_l$ for $-1 \leq l \leq k+1$,
since by assumption, $H$ does not contain edges of the form
$[a_i,a_l]$.

\item Since $y \neq a_l$ for all $-1 \leq l \leq k+1$, $y$ is a
leaf of $B$. Hence, the only edge of $B$ that emanates from $y$
connects it to an interior vertex of the spine of $B$, i.e., is of
the form $[y,a_j]$ for some $0 \leq j \leq k$.

\item Define $f(i)=j$, where $j$ is determined by the two previous
steps.
\end{enumerate}

Note that we indeed have $0 \leq f(i) \leq k$ for all $i$, since
$a_{-1}$ and $a_{k+1}$ are leaves of $B$, and thus are connected
only to $a_0$ and $a_k$, respectively. Also note that $f(i) \neq
i$ for all $i$, since otherwise, $B$ and $H$ would share the edge
$[a_i,y]$, for some $y$.

Consequently, we have $f(0)>0$ and $f(k)<k$, and thus, there
exists an $i$, $0 \leq i \leq k$, such that $j=f(i)>i$ and
$h=f(i+1)<i+1$. Denote the vertices that were used in the
generation of $f(i)$ and of $f(i+1)$ by $y$ and $x$, respectively,
as illustrated in Figure~\ref{Fig:Converse}. We claim that the
edges $[a_i,y]$ and $[a_{i+1},x]$ cross, which contradicts the
assumption that $H$ is simple.

In order to prove this claim, consider the following polygon:
\[
P= \langle
x,a_h,a_{h+1},\ldots,a_{i-1},a_i,a_{i+1},\ldots,a_{j-1},a_j,y,x
\rangle.
\]
(Note that the highest possible value of $h$ is $i$. If $h=i$,
then there are no edges between $a_h$ and $a_i$. Similarly for
$a_{i+1}$ and $a_{j}$.) We claim that the path $P_0=\langle x,
a_h,\ldots,a_i,a_{i+1},\ldots,a_j,y \rangle$ lies on the boundary
of $\mathrm{conv}(P)$. Indeed, all the edges of $P_0$ except for
$[x,a_h]$ and $[a_{j},y]$ lie on the boundary of
$\mathrm{conv}(G)$, so they clearly support $P$. The edge
$[a_h,x]$ also supports $P$, since the line $l(a_h,x)$ meets the
path $\langle a_h,\ldots,a_i,a_{i+1},\ldots,a_j,y \rangle$ only at
its endpoint $a_h$. For similar reasons, $[a_j,y]$ must also
support $P$. It follows that the path $P_0$ is part of the
boundary of $\mathrm{conv} (P)$.

Finally, $a_{i+1}$ lies on the boundary of $\mathrm{conv} (P)$
strictly between $a_i$ and $y$, and $a_i$ lies on the boundary of
$\mathrm{conv} (P)$ strictly between $a_{i+1}$ and $x$. This
implies that the two edges $[a_i,y]$ and $[a_{i+1},x]$ must cross.
This contradicts the assumption that $H$ is simple, and thus
completes the proof of the theorem.
\end{proof}


\begin{thebibliography}{99}
\bibitem{Caterpillar0} F. Harary and A.J. Schwenk, Trees with
Hamiltonian Square, {\it Mathematika} \textbf{18} (1971),
pp.~138–-140.

\bibitem{Caterpillar1} F. Harary and A.J. Schwenk, The Number of Caterpillars,
{\it Disc. Math.} \textbf{6} (1973), pp.~359--365.

\bibitem{Hernando} M. Carmen Hernando, Complejidad de Estructuras
Geom$\mathrm{\acute{e}}$tricas y Combinatorias, Ph.D. Thesis,
Universitat Polit$\mathrm{\acute{e}}$ctnica de Catalunya, 1999 (in
Spanish). Available online at:
http://www.tdx.cat/TDX-0402108-120036/

\bibitem{Pach} G. K$\mathrm{\acute{a}}$rolyi, J. Pach, and
G. T$\mathrm{\acute{o}}$th, Ramsey-type Results for Geometric
Graphs I, {\it Discrete and Computational Geometry}, \textbf{18}
(1997), pp.~247--255.

\bibitem{Maamar-Master} C.~Keller and M. A.~Perles, On the
Smallest Sets Blocking Simple Perfect Matchings in a Convex
Geometric Graph, {\it Israel J. of Math.}, in press. Available
on-line at:
http://arxiv.org/abs/0911.3350.

\bibitem{Co-blockers} C.~Keller and M. A.~Perles, Characterization of
Co-Blockers for Simple Perfect Matchings in a Convex Geometric
Graph, submitted. Available online at:
http://arxiv.org/abs/1011.5883.

\bibitem{Rivera} J. J.~Montellano-Ballesteros and E.~Rivera-Campo,
On the Heterochromatic Number of Hypergraphs Associated to
Geometric Graphs and to Matroids, submitted. Available online at:
http://arxiv.org/abs/1011.4888.

\bibitem{Pach2} J.~Pach, Geometric Graph Theory, Chapter 10 in:
Handbook of Discrete and Computational Geometry, Second Edition,
J. E. Goodman and J. O'Rourke, eds., CRC Press, Boca Raton, 2004,
pp.~219--238.

\bibitem{Micha} M. A.~Perles, unpublished, 1987.

\end{thebibliography}
\end{document}